\documentclass[12pt]{article}
\input epsf.tex

\usepackage{amsmath}
\usepackage{amsthm}
\usepackage{amsfonts}
\usepackage{amssymb}
\usepackage{graphicx}
\usepackage{latexsym}

\usepackage{amsmath,amsthm,amsfonts,amssymb}

\usepackage{epsfig}

\usepackage{amssymb}
\usepackage[all]{xy}
\xyoption{poly}
\usepackage{fancyhdr}
\usepackage{wrapfig}

\theoremstyle{plain}
\newtheorem{thm}{Theorem}[section]

\newtheorem{lem}[thm]{Lemma}
\newtheorem{cor}[thm]{Corollary}

\theoremstyle{definition}

\newtheorem{notation}{Notation}
\theoremstyle{remark}

\advance \topmargin by -\headheight
\advance \topmargin by -\headsep
\evensidemargin \oddsidemargin
\def\cB{{\mathcal B}}

\def\bnd{{\bf bnd}}

\def\cC{{\mathcal C}}

\def\cc{{\curvearrowright}}

\def\E{{\mathbb E}}

\def\cK{{\mathcal{K}}}

\def\cM{{\cal M}}

\def\N{{\mathbb N}}
\def\cN{{\mathcal N}}

\def\P{{\mathbb P}}
\def\Pr{{\textrm{Pr}}}
\def\Prob{{\textrm{Pr}}}

\def\R{{\mathbb R}}

\def\tsigma{\tilde \sigma}

\def\tS{\widetilde{ \Sub_G}}
\def\shadow{{\textrm{Shd}}}
\def\Schreier{{\textrm{Schreier}}}

\def\Sub{{\textrm{Sub}}}

\def\Tree{\textrm{Tree}}

\def\chix{{\raise.5ex\hbox{$\chi$}}}

\def\Z{{\mathbb Z}}

\begin{document}
\title{Random walks on random coset spaces with applications to Furstenberg entropy}
\author{Lewis Bowen\footnote{supported in part by NSF grant DMS-0968762, NSF CAREER Award DMS-0954606 and BSF grant 2008274} \\ Texas A\&M University}
\maketitle
\begin{abstract}
We determine the range of Furstenberg entropy for stationary ergodic actions of nonabelian free groups by an explicit construction involving random walks on random coset spaces.
\end{abstract}


\noindent
\tableofcontents

\section{Introduction}
Let $\mu$ be a Borel probability measure on a locally compact group $G$. An action of $G$ on a probability space $(X,\eta)$ is {\em $\mu$-stationary} if $\eta=\mu*\eta$ where
$$\mu*\eta:= \int g_*\eta~d\mu(g)$$
is the {\em convolution} of $\mu$ with $\eta$. There is significant interest in understanding the structure of stationary actions and  their  connections with random walks \cite{Fu63b, Fu71, Fu72, Fu80}, rigidity theory \cite{NZ99,NZ00,NZ02a, NZ02b, Ne03} and classification of invariant measures \cite{BFLM11, BQ09, BQ11a, BQ11b}. A general structure theory is presented in \cite{FG10}. 

Stationary systems are abundant; indeed every continuous action of $G$ on a compact metric space admits a stationary measure. However tractable examples, other than Poisson boundaries and measure-preserving actions, are somewhat lacking. One of the main contributions of this paper is the construction of new examples.

The {\em Furstenberg entropy} or {\em $\mu$-entropy} of a $\mu$-stationary action of $G$ on a probability space $(X,\eta)$ is a fundamental invariant defined in \cite{Fu63a} by
$$h_\mu(X,\eta):=\iint -\log \frac{d\eta \circ g}{d\eta}(x)~d\eta(x)~d\mu(g).$$

By Jensen's inequality this entropy is always nonnegative. It equals zero if and only if the action is measure-preserving. One of the main results of \cite{NZ00} and \cite{NZ02a} is that if $G$ is a connected higher rank real semisimple Lie group with finite center and the action satisfies a certain mixing hypothesis, then this entropy can take on only a finite number of values corresponding with the actions of $G$ on homogeneous spaces $(G/Q,\nu_Q)$ where $Q < G$ is a parabolic subgroup. Indeed, it is shown that any such $(G,\mu)$-space is a relatively measure-preserving extension of one of these actions. This is a crucial step in Nevo-Zimmer's proof of the generalized intermediate factor theorem, which constitutes a major generalization of Margulis' normal subgroup theorem. 

These results motivate the\\

\noindent {\bf Furstenberg entropy realization problem}: {\em Given $(G,\mu)$ what are all possible values of the $\mu$-entropy $h_\mu(X,\eta)$ as $(X,\eta)$ varies over all ergodic $\mu$-stationary actions of G? }\\

In \cite{NZ00}, page 323, the authors remark that they do not know the full set of possible values of the Furstenberg entropy for a given $(G,\mu)$ or even whether this set of values contains an interval (for any non-amenable group $G$). However, they prove that if $G$ is $PSL_2(\R)$ or a semisimple group of real rank $\ge 2$ containing a parabolic subgroup that maps onto $PSL_2(\R)$ then infinitely many different values are achieved \cite[Theorem 3.4]{NZ00}. It is also proven that if $G$ has property (T) then there is an open interval $(0,\epsilon(\mu))$ containing no values of $h_\mu(X,\eta)$ for any ergodic $\mu$-stationary $G$-systems $(X,\eta)$ \cite{Ne03}. Our main theorem is:

\begin{thm}\label{thm:main}
Let $G=\langle s_1,\ldots, s_r \rangle$ be a free group of rank $2\le r <\infty$, $\mu$ be the uniform probability measure on $\{s_1,\ldots, s_r, s_1^{-1},\ldots, s_r^{-1}\}$ and $h_{max}(\mu)$ denote the maximum value of the $\mu$-entropy over all $\mu$-stationary $G$-actions $(X,\eta)$. Then for every $t \in [0,h_{max}(\mu)]$ there exists an ergodic $\mu$-stationary $G$-action on a probability space $(X,\eta)$ with $h_\mu(X,\eta)=t$. 
\end{thm}

To sketch the proof and explain further results, let us recall the notion of Poisson boundary. So consider a locally compact group $G$ with a probability measure $\mu$ on $G$. Let $X_1,X_2,\ldots$ be a sequence of independent random variables each with law $\mu$. The sequence $\{Z_n\}_{n=1}^\infty$ where $Z_n := X_1\cdots X_n$ is the random walk induced by $\mu$. The {\em Poisson boundary} of this random walk, denoted $(B,\nu)$, is the space of ergodic components of the time shift on $(G^\N, \P)$ where $\P$ is the law of the random walk $\{Z_n\}_{n=1}^\infty$. Because the time shift commutes with the left-action of $G$ on $G^\N$, $G$ acts on the Poisson boundary. This action is $\mu$-stationary. It is well-known that $h_\mu(B,\nu)=h_{max}(\mu)$ (see e.g. \cite[\S 3.2, Corollary 3]{KV83}).

If $K<G$ is a closed subgroup, then we may consider the random walk $\{KZ_n\}_{n=1}^\infty$ on the coset space $K\backslash G$. The Poisson boundary of this random walk is the space $(B_K,\nu_K)$ of ergodic components of the time shift on $( (K\backslash G)^\N, \P_K)$ where $\P_K$ is the law of the random walk $\{KZ_n\}_{n=1}^\infty$. If $K$ is normal in $G$, then $G$ acts on the left on $(K\backslash G)^\N$ and this action descends to an action on $B_K$. Moreover $\nu_K$ is $\mu$-stationary. More generally, if $K$ has only finitely many conjugates $\{K_1,\ldots, K_n\}$  then $G$ acts on the left on $\cup_{i=1}^n (K_i \backslash G)^\N$ and this action descends to an action on the finite union $\cup_{i=1}^n  B_{K_i}$. Moreover this preserves an ergodic stationary measure.

Our second main result is: 

\begin{thm}\label{thm:dense}
Let $(G,\mu)$ be as in Theorem \ref{thm:main}. Then the set of $\mu$-entropies of actions of the form $G \cc ( \cup_{i=1}^n B_{K_i}, \frac{1}{n}\sum_{i=1}^n \nu_{K_i})$ (where $\{K_1,\ldots, K_n\}$ is a conjugacy class of a subgroup of $G$), is dense in $[0,h_{max}(\mu)]$. 
\end{thm}

If $K$ is not normal in $G$ then there is no canonical action of $G$ on $B_K$. To remedy this, consider the space $\Sub_G$ of all closed subgroups of $G$. $G$ acts on this space by conjugation. Let $\cM(\Sub_G)$ denote the space of conjugation-invariant Borel probability measures on $\Sub_G$. A random subgroup with law $\lambda \in \cM(\Sub_G)$ is called an {\em invariant random subgroup} or IRS for short. This term was coined in \cite{ABBGNRS11}. There has been a recent increase in studies of the action of $G$ on $\Sub_G$ and its invariant measures \cite{Bo12, AGV12, Vo12, ABBGNRS11, Ve11, Sa11, Gr11, Ve10, BS06, DS02, GS99, SZ94}. 


For $\lambda \in \cM(\Sub_G)$, we consider the random walk $\{KZ_n\}_{n=1}^\infty$ on the coset space $K \backslash G$ where $K<G$ is random with law $\lambda$ (and $Z_n$ are as above). The Poisson boundary of this random walk is the space $(B(\Sub_G), \nu_\lambda)$ of ergodic components of the time shift on $( \tS, \P_\lambda)$ where $\tS$ is the set of all $(K;Kg_0,Kg_1,\ldots)$ with $K \in \Sub_G$, $g_0,g_1,\ldots \in G$ and $\P_\lambda$ is the law of $(K; Kg_0, Kg_1,\ldots)$. The group $G$ naturally acts on this space and $\nu_\lambda$ is stationary and ergodic if $\lambda$ is ergodic. 

Incidentally, we will prove a few fundamental results about these random walks in the case of an arbitrary countable discrete group $G$. For example, the {\em random walk entropy} of the walk $\{Z_n\}_{n=1}^\infty$ on $G$ is defined to be $\lim_{n\to\infty} n^{-1}H( \mu^n)$ where $\mu^n$ is the $n$-fold convolution power of $\mu$ and $H(\mu^n) = - \sum_{g\in G} \mu^n(\{g\})\log\mu^n(\{g\})$. In \cite{KV83}, Kaimanovich and Vershik proved that the random walk entropy equals the Furstenberg entropy of the associated Poisson boundary. In  \S \ref{sec:KV} this result is generalized to random walks on the coset space of an invariant random subgroup.

The map which takes $\lambda \in \cM(\Sub_G)$ to $h_\mu( B(\Sub_G),\nu_\lambda)$ is not continuous in general. For example, consider a decreasing sequence $\{N_i\}_{i=1}^\infty$ of finite-index normal subgroups with trivial intersection $\cap_{i=1}^\infty N_i = \{e\}$. If $\delta_i \in \cM(\Sub_G)$ is the Dirac measure concentrated on $N_i$ then $(B(\Sub_G), \nu_{\delta_i}) = (B_{N_i}, \nu_{N_i})$. Because $N_i$ has finite index, $h_\mu( B_{N_i}, \nu_{N_i}) = 0$. However, $\delta_i$ converges as $i\to\infty$ to $\delta_e$, the Dirac measure concentrated on the trivial subgroup. Because $h_\mu( B_e, \nu_e) = h_{max}(\mu)>0$, this map is discontinuous. In spite of this discontinuity, we will show that when $G$ is a free group, there exist paths in $\cM(\Sub_G)$ on which entropy varies continuously and use these paths to establish Theorem \ref{thm:main}.

{\bf Acknowledgements:} I'd like to thank Amos Nevo for asking me whether Theorem \ref{thm:main} is true and for several motivating discussions and Yuri Lima, Yair Hartman and Omer Tamuz for discovering an error in a previous version. I'd also like to thank the anonymous referees for their careful readings and helpful criticism.


\section{Poisson boundaries of random walks on coset spaces}\label{sec:general}
Let $G$ be a separable locally compact group with a probability measure $\mu$. We assume $\mu$ is {\em admissible}: its support generates $G$ as a semigroup and some convolution power $\mu^n$ is absolutely continuous with respect to Haar measure on $G$. The purpose of this section is to set notation and define the Poisson boundary of the $\mu$-induced random walk on a coset space $K\backslash G$.

Let $\N:=\{0,1,2,\ldots\}$, $\N_{\ge 1}:=\{1,2,\ldots\}$ and $m:G^{\N} \to G^{\N}$ be the multiplication map
$$m(g_0,g_1,g_2,\ldots) := (g_0,g_0g_1,g_0g_1g_2,\ldots).$$
For $g\in G$, define a probability measure on $G^{\N}$ by $\P_g:=m_*(\delta_g \times \mu^{\N_{\ge 1}})$ where $\delta_g$ is the Dirac probability measure concentrated on $\{g\} \subset G$. We write $\P$ to denote $\P_e$ where $e$ is the identity element.

Let $K<G$ be a closed subgroup and $\pi_K:G^{\N} \to (K\backslash G)^{\N}$ the quotient map 
$$\pi_K(g_0,g_1,g_2,\ldots):=(Kg_0, Kg_1,Kg_2,\ldots).$$
$\P_{Kg} := (\pi_K)_*\P_g$ denotes the pushforward measure. Of course, $\P_K:=\P_{Ke}$.

Note that $G$ acts on $G^\N$ on the left by $g(g_0,g_1,\ldots)=(gg_0,gg_1,\ldots)$. This action commutes with the shift $\sigma:G^{\N} \to G^{\N}$ defined by:
$$\sigma(g_0,g_1,g_2, \ldots) := (g_1,g_2, \ldots).$$
Let $\cB(\sigma)$ be the sigma-algebra of $\sigma$-invariant Borel subsets of $G^\N$. By Mackey's Point Realization Theorem \cite[Theorem 1]{Ma62}, there exists a standard Borel probability space $(B_e,\nu_e)$, a Borel action of $G$ on $B$ and a $G$-equivariant Borel map $\bnd: (G^\N)' \to B'$ (where $(G^\N)' \subset G^\N$ and $B' \subset B_e$ are conull) such that the inverse image of the Borel sigma-algebra on $B_e$ equals $\cB(\sigma)$ (modulo sets of measure zero). The space $(B_e,\nu_e)$ is called the {\em Poisson boundary} of $(G,\mu)$. 


Similarly, let $\sigma_K:(K\backslash G)^{\N} \to (K\backslash G)^{\N}$ be the shift map:
$$\sigma_K(Kg_0,Kg_1,Kg_2, \ldots) := (Kg_1,Kg_2, \ldots).$$
Denote the sigma-algebra of shift-invariant Borel subsets of $(K\backslash G)^{\N}$ by $\cB(\sigma_K)$. Mackey's Point Realization Theorem implies the existence of a standard Borel space $B_K$, a probability measure $\nu_K$ on $B_K$ and a Borel map $\bnd_K: ((K\backslash G)^\N)' \to B_K'$ (where $((K\backslash G)^\N)' \subset (K\backslash \N), B_K' \subset B_K$ are conull) such that $(\bnd_K)_*\P_K=\nu_K$, $\cB(\sigma_K)$ is the pullback of the Borel sigma-algebra on $B_K$ (modulo sets of measure zero). Let $\nu_{Kg} := (\bnd_K)_*\P_{Kg}$ be the pushforward measure on $B_K$ (for any $g\in G$). Then $(B_K, \nu_K)$ is the {\em Poisson boundary} of $K\backslash G$ generated by $\mu$.

The commutative diagram:
\begin{displaymath}
\xymatrix{(G^{\N}, \P) \ar[d]^{\pi_K} \ar[r]^{\bnd} & (B_e, \nu_e)  \ar[d]^{\pi_K} \\
((K\backslash G)^{\N}, \P_K)   \ar[r]^{\bnd} & (B_K, \nu_K) }
\end{displaymath}
uses an abuse of notation: we let $\pi_K$ denote the map from $G^{\N}$ to $(K\backslash G)^{\N}$ as well as the induced map from $B_e$ to $B_K$. Also we let $\bnd$ denote the map from (a conull subset of) $G^{\N}$ to $B_e$ as well as the map from (a conull subset of) $(K\backslash G)^{\N}$ to $B_K$ when no confusion can arise.




\subsection{The space of subgroups}\label{sec:space}

The group $G$ acts on the set of its closed subgroups $\Sub_G$ by conjugation. The set $\Sub_G$ with the topology of uniform convergence on compact subsets is a compact metrizable space. Let $\cM(\Sub_G)$ be the space of all conjugation-invariant Borel probability measures on $\Sub_G$.

Let $\tS = \{ (K; Kg_0, Kg_1, Kg_2,\ldots):~ K\in \Sub_G, g_0,g_1, g_2,\ldots \in G\}$. If $2^G$ denotes the space of closed subsets of $G$ then $\tS$ naturally embeds into the product space $(2^G)^\N$ by $(K;Kg_0,Kg_1,\ldots) \mapsto (K,Kg_0,Kg_1,\ldots)$. The space $2^G$ is compact under the topology of uniform convergence on compact subsets. By Tychonoff's Theorem, $(2^G)^\N$ is also compact and therefore, since $\tS$ is closed as a subset of $(2^G)^\N$, it is also compact under the subspace topology. Given an invariant measure $\lambda \in \cM(\Sub_G)$, let $\P_\lambda$ be the measure on $\tS$ whose fiber over $K \in \Sub_G$ is $\P_K$:
$$d\P_\lambda(K; Kg_0, Kg_1, Kg_2,\ldots)= d\P_K(Kg_0,Kg_1,\ldots) d\lambda(K).$$
The group $G$ acts on $\tS$ by
$$\gamma (K; Kg_0, Kg_1,\ldots) := (K^\gamma; \gamma K g_0, \gamma K g_1,\ldots)\quad \forall \gamma, g_0, g_1, \in G, K \in \Sub_G.$$
This action commutes with the shift action $\tsigma:\tS \to \tS$ which is defined by
$$\tsigma(K; Kg_0, Kg_1,\ldots) := (K; Kg_1,\ldots).$$
Let $\cB(\tilde{\sigma})$ denote the sigma-algebra of $\tsigma$-invariant Borel subsets of $\tS$. By Mackey's Point Realization Theorem, there exists a standard Borel probability space $(B(\Sub_G), \nu_\lambda)$, a nonsingular $G$-action on $B(\Sub_G)$ and a $G$-equivariant Borel map $\bnd:\tS' \to B(\Sub_G)$ (where $\tS' \subset \tS$ is conull) such that $\bnd_*\P_\lambda=\nu_\lambda$ and $\cB(\tsigma)$ is the pullback of the Borel sigma-algebra on $B(\Sub_G)$ (up to sets of measure zero). By Mackey's Point Realization Theorem again, for any $K \in \Sub_G$ there is a Borel map $\phi:B_K' \to B(\Sub_G)$ (where $B'_K \subset B_K$ is conull) such that $\phi(B'_K) = \bnd(\{ (K; Kg_0, Kg_1, Kg_2,\ldots):~ g_0,g_1, g_2,\ldots \in G\})$ and $d\nu_\lambda(\xi) = d(\phi_*\nu_K)(\xi) d\lambda(K)$. By abuse of notation, we identify $(B_K,\nu_K)$ with its image under $\phi$. Thus we write $d\nu_\lambda(\xi) = d\nu_K(\xi) d\lambda(K)$.





We have the following commutative diagram:
\begin{displaymath}\label{diagram2}
\xymatrix{  ( \Sub_G \times G^{\N}, \lambda \times \P)  \ar[d]^{\pi} \ar[r]^{\bnd} & (\Sub_G \times B_e, \lambda \times \nu_e) \ar[d]^{\pi} \\
  (\tS, \P_\lambda)\ar[r]^{\bnd} & (B(\Sub_G), \nu_\lambda) }
\end{displaymath}
By abuse of notation we let $\bnd$ denote both the map from $\Sub_G \times G^{\N}$ to $\Sub_G \times B_e$ which takes $(K; g_0,g_1,\ldots)$ to $(K, \bnd( g_0,g_1,\ldots))$ as well as the map from $\tS$ to  $B(\Sub_G)$. We also let $\pi$ denote both the map from $\Sub_G \times G^{\N}$ to $\tS$ which takes $(K; g_0,g_1,\ldots)$ to $(K; Kg_0, Kg_1, \ldots)$ as well as the induced map from (a conull subset of) $\Sub_G \times B_e$ to  $B(\Sub_G)$.


\begin{lem}
If $\lambda \in \cM(\Sub_G)$ is ergodic for the $G$-action on $\Sub_G$ then $\nu_\lambda$ is also ergodic for the $G$-action on $B(\Sub_G)$. Moreover $\nu_\lambda$ is $\mu$-stationary.
\end{lem}

\begin{proof}
From the diagram above, it follows that $G \cc (\Sub_G \times B_e, \lambda \times \nu_e)$ factors onto $G\cc (B(\Sub_G),\nu_\lambda)$. Because the Poisson boundary $G \cc (B_e,\nu_e)$ is weakly mixing [AL05] and $\lambda$ is ergodic, $G \cc (\Sub_G \times B_{e}, \lambda \times \nu_{e})$ is ergodic. Since  $G \cc (B(\Sub_G),\nu_\lambda)$ is a factor of an ergodic system, it is also ergodic. The measure $\lambda \times \nu_{e}$ is stationary since $\lambda$ is invariant and $\nu_{e}$ is stationary. Since  $G \cc (B(\Sub_G),\nu_\lambda)$ is a factor of a stationary system, it is also stationary.
\end{proof}

\section{Entropy formulae}\label{sec:KV}
In this section, we require $G$ to be a countable discrete group with an admissible measure $\mu$. Our goal in this section is to provide a formula for the $\mu$-entropy in terms of the so-called random walk entropy. To explain, we need a few definitions.

\begin{notation}
To simplify notation, for any $\lambda \in \cM(\Sub_G)$, let $h_\mu(\lambda):=h_\mu(B(\Sub_G), \nu_\lambda)$.
\end{notation}

We let $\mu^n$ be the $n$-fold convolution of $\mu$. In other words, if $m_n:G^n \to G$ denotes the multiplication map
$$m_n(g_1,g_2,\ldots,g_n)=g_1g_2\cdots g_n$$
and $(G^n,(\times \mu)^n)$ denotes the direct product of $n$ copies of $(G,\mu)$ then $\mu^n=(m_n)_*(\times \mu)^n$. 

For $K\in \Sub_G$, let $\mu^n_{K}$ be the measure on $K\backslash G$ given by $\mu^n_K:=(\pi_K)_*\mu^n$ where $\pi_K:G \to (K\backslash G)$ is the quotient map. Similarly, if $g,h \in G$ then $\mu^n_{gKh}$ is the measure on $gKg^{-1}\backslash G$ given by 
$$\mu^n_{gKh}( E) = \mu^n(\{ \gamma \in G:~ gKh\gamma \in E\})\quad \forall E \subset gKg^{-1}\backslash G.$$

In general, if $\omega$ is a probability measure on a finite or countable set $W$ then the {\em entropy} of $\omega$ is
$$H(\omega):=-\sum_{w \in W} \omega(\{w\}) \log( \omega(\{w\}))$$
where by convention $0\log(0)=0$.

The sequence $\{H(\mu^n)\}_{n=1}^\infty$ can be shown to be sub-additive. Therefore the limit of $\frac{H(\mu^n)}{n}$ as $n\to\infty$ exists. This limit is called the {\em random walk entropy of $(G,\mu)$}. In (\cite{KV83}, Theorem 3.1), it is shown that this coincides with the $\mu$-entropy of the Poisson boundary $(B_e,\nu_e)$. Analogously, the main result of this section is:
\begin{thm}\label{thm:hard}
Suppose $H(\mu)<\infty$. Then for any invariant measure $\lambda \in \cM(\Sub_G)$,
\begin{eqnarray*}
h_\mu(\lambda) &=& \lim_{n\to\infty} \frac{1}{n} \int H(\mu^n_K)~d\lambda(K) = \inf_{n} \frac{1}{n} \int H(\mu^n_K)~d\lambda(K) \\
 &=&\lim_{n\to\infty} \int \left(H(\mu^n_K)- H(\mu^{n-1}_K)\right)~d\lambda(K)=\inf_{n\to\infty} \int \left(H(\mu^n_K)- H(\mu^{n-1}_K)\right)~d\lambda(K).
  \end{eqnarray*}
\end{thm}


For $y \in G^{\N}$ or $y\in (K\backslash G)^{\N}$ we let $y_n$ be the $n$-coordinate of $y$. So $y=(y_0,y_1,\ldots)$. We let $\alpha_n$ be the partition of $G^{\N}$ determined by the condition that $y,y'$ are in the same partition element if and only if $y'_i =y_i$ for $0\le i \le n$. We let $\eta_n$ be the partition of $G^{\N}$ determined by the condition that $y,y'$ are in the same partition element if and only if $y'_i =y_i$ for $i \ge n$. We let $\tau_n$ be the partition of $G^{\N}$ determined by the condition that $y,y'$ are in the same partition element if and only if $y'_n =y_n$. We define the partitions $\alpha^K_n,\eta_n^K, \tau_n^K$ of $(K\backslash G)^{\N}$ similarly. We let $\alpha_n^K(y)$ denote the partition element of $\alpha_n^K$ that contains $y$ (and similar notation holds for the other partitions).

Given partitions $\alpha,\beta$ of a probability space $(X,\kappa)$, the entropy of $\alpha$ relative to $\beta$ is:
$$H(\alpha|\beta):=-\int \log\left( \kappa( \alpha(x)|\beta(x))\right)~d\kappa(x)$$
where $\alpha(x)$ denotes the partition element of $\alpha$ containing $x$ and $ \kappa( \alpha(x)|\beta(x)) = \frac{\kappa(\alpha(x) \cap \beta(x))}{\kappa(\beta(x))}$. We assume throughout the rest of this section that $H(\mu)<\infty$ and let $\lambda \in \cM(\Sub_G)$ be fixed.
\begin{lem}
For any $K \in \Sub_G$,
$$\int H(\alpha^K_1|\eta^K_n) ~d\lambda(K) = \int \left(H(\mu_K) - H(\mu^n_K) + H(\mu_{K}^{n-1})\right)~d\lambda(K).$$
\end{lem}

\begin{proof}
The Markov property implies $\P_K(\alpha^K_1(y) | \eta^K_n(y))  = \P_K(\alpha^K_1(y) | \tau^K_n(y))$ for any $K \in \Sub_G$ and $y\in (K\backslash G)^{\N}$. So
\begin{eqnarray*}
\P_K(\alpha^K_1(y) | \eta^K_n(y)) &=&  \P_K(\alpha^K_1(y) | \tau^K_n(y))\\
&=& \frac{\P_K(\{y' \in (K\backslash G)^{\N}:~y'_1=y_1,~y'_n=y_n\})}{\P_K(\{y' \in (K\backslash G)^{\N}:~y'_n=y_n\})}\\
 &=& \frac{ \mu_K(y_1) \P_K(\tau^K_n(y)|\alpha^K_1(y))}{\mu_K(y_n)}.
\end{eqnarray*}
Note $$ \P_K(\tau^K_n(y)|\alpha^K_1(y)) =\mu^{n-1}_{y_1}(y_n).$$
We now have:
\begin{eqnarray}
\P_K(\alpha^K_1(y) | \eta^K_n(y)) =\frac{ \mu_K(y_1)\mu^{n-1}_{y_1}(y_n)}{\mu_K^n(y_n)}.
\end{eqnarray}

Therefore,
\begin{eqnarray}
H(\alpha^K_1|\eta^K_n) &=& -\int \log\left( \P_K(\alpha^K_1(y) | \eta^K_n(y))\right)~d\P_K(y) \\\label{eqn:start} 
&=& H(\mu_K) - H(\mu^n_K) + \sum_{g\in G} \mu(g)H(\mu_{Kg}^{n-1}).
\end{eqnarray}
Since $\lambda$ is conjugation-invariant and $H(\mu_{Kg}^{n-1}) = H(\mu_{g^{-1}Kg}^{n-1})$,
\begin{eqnarray*}
\int H(\mu_{Kg}^{n-1})~d\lambda(K)=  \int H(\mu_{g^{-1}Kg}^{n-1})~d\lambda(K)  = \int H(\mu_{K}^{n-1})~d\lambda(K).
  \end{eqnarray*}
So (\ref{eqn:start}) implies
$$\int H(\alpha^K_1|\eta^K_n) ~d\lambda(K) = \int \left(H(\mu_K) - H(\mu^n_K) + H(\mu_{K}^{n-1})\right)~d\lambda(K).$$
\end{proof}

Let $\eta^K$ be the limit of $\eta^K_n$ (so a set $E$ is in the $\sigma$-algebra generated by $\eta^K$ iff for every $n$ it is in the $\sigma$-algebra generated by $\eta^K_n$). 

\begin{lem}\label{lem:11}
The sequence $\int \left(H(\mu^n_K) - H(\mu^{n-1}_K)\right) ~d\lambda(K)$ is monotone decreasing in $n$. Therefore, 
\begin{eqnarray*}
\lim_{n\to\infty} \frac{1}{n} \int H(\mu^n_K)~d\lambda(K) &=&\inf_{n\to\infty} \frac{1}{n} \int H(\mu^n_K)~d\lambda(K)\\
= \lim_{n\to\infty} \int \left(H(\mu^n_K) - H(\mu^{n-1}_K)\right) ~d\lambda(K)&=& \inf_{n\to\infty} \int \left(H(\mu^n_K) - H(\mu^{n-1}_K)\right) ~d\lambda(K)\\
 &=&\int \left(H(\mu_K) - H(\alpha^K_1|\eta^K)\right)~d\lambda(K).
\end{eqnarray*}
\end{lem}

\begin{proof} 
Since $\eta^K_{n-1}$ refines $\eta^K_{n}$ we have $H(\alpha^K_1|\eta^K_n) \ge H(\alpha^K_1|\eta^K_{n-1})$. So the previous lemma implies $\int H(\mu^n_K) - H(\mu^{n-1}_K) ~d\lambda(K)$ is monotone decreasing in $n$. It is also bounded by $H(\mu)$. So,
\begin{eqnarray*}
\lim_{n\to\infty}\frac{1}{n} \int H(\mu_K^n)~d\lambda(K) &=&  \lim_{n\to\infty} \frac{1}{n} \sum_{m=1}^n \int \left(H(\mu_K^m) - H(\mu_K^{m-1})\right)~d\lambda(K)\\
 &=& \lim_{n\to\infty} \int \left(H(\mu_K^n) - H(\mu_K^{n-1})\right)~d\lambda(K)\\
&=& \lim_{n\to\infty} \int \left(H(\mu_K) - H(\alpha^K_1|\eta^K_n)\right) ~d\lambda(K)\\
&=& \int \left(H(\mu_K) - H(\alpha^K_1|\eta^K)\right)~d\lambda(K).
\end{eqnarray*}
The third line follows from the previous lemma.

Because $\int \left(H(\mu^n_K) - H(\mu^{n-1}_K)\right) ~d\lambda(K)$ is monotone decreasing, it follows that
\begin{eqnarray*}
\frac{1}{n} \int H(\mu_K^n)~d\lambda(K) &=& \frac{1}{n} \sum_{m=1}^n \int \left(H(\mu_K^m) - H(\mu_K^{m-1})\right)~d\lambda(K)\\
&\ge&  \int \left(H(\mu_K^n) - H(\mu_K^{n-1})\right)~d\lambda(K).
\end{eqnarray*}
Therefore,
$$\lim_{n\to\infty} \frac{1}{n} \int H(\mu^n_K)~d\lambda(K) = \inf_n \int \left(H(\mu_K^n) - H(\mu_K^{n-1})\right)~d\lambda(K) = \inf_n \frac{1}{n} \int H(\mu_K^n)~d\lambda(K).$$
 \end{proof}

\begin{lem}\label{lem:12}
For any $K \in \Sub_G$,
$$H(\mu_K) - H(\alpha^K_1|\eta^K)= \sum_{Kg\in K\backslash G} \mu_K(Kg) \int \log\left(\frac{d\nu_{Kg}}{d\nu_K}(b)\right) ~d\nu_{Kg}(b).$$
\end{lem}

\begin{proof}
For any Borel $E \subset B_K$ and any $y \in (K\backslash G)^{\N}$,
$$\P_K\left( \{y'  \in (K\backslash G)^{\N}:~ \bnd(y') \in E\}|~ \alpha_1^K(y)\right) = \nu_{y_1}(E) = \int_E \frac{d\nu_{y_1}}{d\nu_K}(b)~d\nu_K(b).$$
Therefore,
$$\P_K(\alpha^K_1(y)|~\eta^K(y)) = \P_K(\alpha_1^K(y))\frac{d\nu_{y_1}}{d\nu_K}(\bnd(y))$$
for $\P_K$ a.e. $y$. We now have:
\begin{eqnarray*}
H(\mu_K) - H(\alpha^K_1|\eta^K)&=& H(\mu_K) + \int \log\left(\P_K(\alpha^K_1(y)|~\eta^K(y))\right)~d\P_K(y) \\
&=& H(\mu_K) + \int \log\left(\P_K(\alpha_1^K(y))\frac{d\nu_{y_1}}{d\nu_K}(\bnd(y))\right) ~d\P_K(y)\\
&=& \int \log\left(\frac{d\nu_{y_1}}{d\nu_K}(\bnd(y))\right) ~d\P_K(y)\\
&=& \sum_{Kg\in K\backslash G} \int_{\{y:~y_1=Kg\}} \log\left(\frac{d\nu_{Kg}}{d\nu_K}(\bnd(y))\right) ~d\P_K(y)\\
&=& \sum_{Kg\in K\backslash G} \mu_K(Kg) \int \log\left(\frac{d\nu_{Kg}}{d\nu_K}(b)\right) ~d\nu_{Kg}(b).
\end{eqnarray*}
\end{proof}

\begin{lem}\label{lem2}
For any $K \in \Sub_G$, Borel set $E \subset B_K$ and $\gamma \in G$,
$$\nu_K(E) = \nu_{\gamma K}(\gamma E).$$
\end{lem}
\begin{proof}
The proof is immediate.
\end{proof}

\begin{lem}\label{lem:rn}
For $\gamma\in G$ and $\xi \in B_K \subset {B(\Sub_G)}$,
$$\frac{d\nu_\lambda\circ \gamma^{-1}}{d\nu_\lambda}(\xi) = \frac{d\nu_{K\gamma}}{d\nu_K}(\xi).$$
\end{lem}
\begin{proof}
For any Borel $E \subset {B(\Sub_G)}$,
\begin{eqnarray*}
\nu_\lambda\circ \gamma^{-1}(E) &=& \nu_\lambda(\gamma^{-1} E) = \int \nu_K(\gamma^{-1} E \cap B_K) ~d\lambda(K).
\end{eqnarray*}
By Lemma \ref{lem2}, 
$$\nu_K(\gamma^{-1} E \cap B_K)  = \nu_{\gamma K}(E \cap B_{K^\gamma }).$$
So
$$\nu_\lambda\circ \gamma^{-1}(E)= \int   \nu_{\gamma K}(E \cap B_{K^\gamma})  ~d\lambda(K).$$
Make the change of variable $L=K^\gamma$ and use the conjugation-invariance of $\lambda$ to obtain
$$\nu_\lambda\circ \gamma^{-1}(E)= \int  \nu_{L\gamma }(E \cap B_{L}) ~d\lambda(L).$$
In other words,
$$\nu_\lambda\circ \gamma^{-1}(E)= \iint \frac{d\nu_{K\gamma}}{d\nu_K}(\xi) 1_E(\xi)~d\nu_K(\xi) d\lambda(K).$$
This implies the lemma.
\end{proof}


\begin{proof}[Proof of Theorem \ref{thm:hard}]
By Lemmas \ref{lem:11}, \ref{lem:12} and \ref{lem:rn},
\begin{eqnarray*}
\lim_{n\to\infty} \frac{1}{n} \int H(\mu^n_K)~d\lambda(K) &=&\int H(\mu_K) - H(\alpha^K_1|\eta^K)~d\lambda(K)\\
&=& \int \sum_{Kg\in K\backslash G} \mu_K(Kg) \int \log\left(\frac{d\nu_{Kg}}{d\nu_K}(b)\right) ~d\nu_{Kg}(b)d\lambda(K)\\
&=& \sum_{g\in G} \mu(g) \iint \log\left(\frac{d\nu_{Kg}}{d\nu_K}(b)\right) \frac{d\nu_{Kg}}{d\nu_K}(b)~d\nu_K(b)d\lambda(K)\\
&=& \sum_{g\in G} \mu(g) \iint \log\left(\frac{d\nu_\lambda\circ g^{-1}}{d\nu_\lambda}(b)\right) \frac{d\nu_\lambda\circ g^{-1}}{d\nu_\lambda}(b)~d\nu_\lambda(b).
\end{eqnarray*}
The cocycle identity for the Radon-Nikodym derivative implies 
$$\frac{d\nu_\lambda\circ g^{-1}}{d\nu_\lambda}(b) = \frac{d\nu_\lambda}{d\nu_\lambda \circ g}(g^{-1}(b)).$$
By definition, we also have 
$$\frac{d\nu_\lambda\circ g^{-1}}{d\nu_\lambda}(b)~d\nu_\lambda(b) = d\nu_\lambda\circ g^{-1}(b).$$
Therefore,
\begin{eqnarray*}
\lim_{n\to\infty} \frac{1}{n} \int H(\mu^n_K)~d\lambda(K) &=& -\sum_{g\in G} \mu(g) \iint \log\left(\frac{d\nu_\lambda\circ g}{d\nu_\lambda}(g^{-1}b)\right) d\nu_\lambda(g^{-1}b)\\
&=&-\sum_{g\in G} \mu(g) \iint \log\left(\frac{d\nu_\lambda \circ g}{d\nu_\lambda}( b)\right) d\nu_\lambda(b)\\
&=&h_\mu(\lambda).
\end{eqnarray*}
The other equalities follow from Lemma \ref{lem:11}. 
\end{proof}


\section{Results for the free group}\label{sec:specific}

For the sake of simplicity, we specialize to the case $G=\langle a,b\rangle$, the rank $2$ free group although all the constructions easily generalize to any finitely generated free group.


Let $\Schreier(K\backslash G)=(V_K,E_K)$ be the Schreier coset graph of $K\backslash G$. The vertex set is $V_K := K\backslash G$. For each $Kg \in K \backslash G$ there are two directed labeled edges in the edge set, denoted by $E_K$. These are $(Kg,Kga)$ which is labeled $a$, and $(Kg,Kgb)$ which is labeled $b$. It is possible that $Kga=Kgb$ in which case there are two different edges from $Kg$ to $Kga=Kgb$. 

We say that $K\backslash G$ is {\em tree-like} if for every $Kg, Kg' \in K\backslash G$ there is a unique sequence of vertices $Kg=Kg_1 , Kg_2,\ldots, Kg_n = Kg'$ such that $Kg_i$ is adjacent to $Kg_{i+1}$ for $1\le i <n$ and $Kg_{i-1} \notin \{Kg_i, Kg_{i+1}\}$ for any $1 < i < n$. This does not mean that there is a unique path in the Schreier coset graph of $K \backslash G$ because it is possible, for example, that $Kga=Kgb$ for some coset $Kg$. Equivalently, $K\backslash G$ is treelike if it does not contain simple circuits of length greater than 2.

Let $\Tree_G \subset \Sub_G$ be the set of all subgroups $K \in \Sub_G$ such that $K \backslash G$ is tree-like. This is a closed $G$-invariant subspace. Let $\cM(\Tree_G) \subset \cM(\Sub_G)$ denote those measures with support contained in $\Tree_G$.

Let $\{X_i\}_{i=1}^\infty$ be i.i.d. random variables in $G$ with law $\mu$ (where $\mu$ is the uniform probability measure on $\{a,a^{-1},b,b^{-1}\}$). For $K \in \Sub_G$, let $R_n(\mu,K)$ be the probability that $KX_1\cdots X_n = K$ and let $R_{\ge n}(\mu,K)$ be the probability that $KX_1\cdots X_m=K$ for some $m\ge n$. A subset $\cN \subset \cM(\Sub_G)$ has {\em controlled return-time probabilities} if 
$$ \lim_{n \to \infty} \sup_{\eta \in \cN} \eta \Big( \{K \in \Sub_G:~R_{\ge n}(\mu,K) \ge \epsilon\} \Big) =0 \quad \forall \epsilon>0.$$

The next result plays a key role in the proof of Theorem \ref{thm:main}. It is proven in the next subsection.

\begin{thm}\label{thm:continuous}
If $\cN \subset \cM(\Tree_G)$ is a set of measures with controlled return-time probabilities then the entropy function $\lambda \in \cN \mapsto h_\mu(\lambda)$ is continuous on $\cN$ with respect to the weak* topology. \end{thm}

\subsection{A continuity criterion }

If $K \in \Tree_G$ (so $K \backslash G$ is tree-like), then for each $g \in G$, let the {\em shadow} of $Kg$, denoted $\shadow(Kg)$, be the set of all cosets $K\gamma \in K \backslash G$ so that every path in $\Schreier(K\backslash G)$ from $K$ to $K\gamma$ passes through $Kg$. Let $\shadow_{\N_{\ge 1}}(Kg)$ be the set of all sequences $(Kg_0, Kg_1, \ldots) \in (K\backslash G)^{\N}$ that are eventually in $\shadow(Kg)$ in the sense that there exists an $N$ so that if $n\ge N$ then $Kg_n \in \shadow(Kg)$. Let $B_{Kg} := \pi_K(\shadow_{\N_{\ge 1}}(Kg))$ be the projection of $\shadow_{\N_{\ge 1}}(Kg)$ to the boundary $B_K \subset B(\Sub_G)$. 



\begin{lem}\label{lem:rn2}
Let $\lambda \in \cM(\Tree_G)$ and $s,t \in \{a,b,a^{-1},b^{-1}\}$.  Then for $\nu_\lambda$-a.e. $\xi$, if $\xi \in B_{Ks}$ then
$$\frac{d \nu_\lambda\circ t}{d\nu_\lambda}(\xi) = \frac{\nu_{Kt^{-1}}(B_{Ks})}{\nu_K(B_{Ks})}.$$
\end{lem}

\begin{proof}
If the random walk on $K\backslash G$ is recurrent then $B_K$ is trivial and the statement is obvious. So we will assume that for $\lambda$-a.e. $K$, $K\backslash G$ is transient. By Lemma \ref{lem:rn}, 
$$\frac{d \nu_\lambda \circ t}{d\nu_\lambda}(\xi) = \frac{d\nu_{Kt^{-1}}}{d\nu_K}(\xi) \quad \textrm{for a.e. } \xi \in B_K \subset B(\Sub_G).$$
Let $\{X_n(Kg):~ Kg \in K \backslash G, n\ge 1\}$
be an i.i.d. family of random variables with law $\mu$. Let $\{Z_n(Kg)\}_{n=1}^\infty$ be the random walk:
$$Z_n(Kg):=KgX_1(Kg)X_2(Kg)\cdots X_n(Kg)$$
and $Z_0(Kg):=Kg$. 

Recall that $\bnd_K$ denotes the projection from the space of sequences $(K\backslash G)^{\N}$ to the boundary $B_K$. Let $\zeta(Kg)=\bnd_K(\{Z_n(Kg)\}_{n=0}^\infty)$. 

Suppose that $Kt^{-1} \ne Ks$. Then any path in $\Schreier(K\backslash G)$ from $Kt^{-1}$ whose projection lies in $B_{Ks}$ necessarily passes through $K$. So for any Borel $E \subset B_{Ks}$ the probability that $\zeta(Kt^{-1}) \in E$ is
\begin{eqnarray*}
&&\nu_{Kt^{-1}}(E) = \Prob(\zeta(Kt^{-1}) \in E)\\
&=&\sum_{n=0}^\infty \Prob\left(Z_n(Kt^{-1}) = K,~  Z_m(Kt^{-1}) \ne K~\forall m>n, \zeta(Kt^{-1}) \in E\right)\\
 & = & \sum_{n=0}^\infty \Prob\left(Z_n(Kt^{-1}) = K\right) \cdot \Prob\left(\zeta(K) \in E,~ Z_t(K) \ne K ~\forall t>0\right)\\
 &=&  \E\left[ |\{ n \ge 0:~ Z_n(Kt^{-1}) =K \}|\right] \cdot \Prob\left(\zeta(K) \in E \textrm{ and } Z_t(K) \ne K~ \forall t>0\right)\\
  &=&  \E\left[ |\{ n \ge 0:~ Z_n(Kt^{-1}) =K \}|\right] \cdot \frac{\Prob (\zeta(K) \in E ) }{\sum_{t=0}^\infty \Prob(Z_t(K)= K)}\\
  &=&\nu_{K}(E) \frac{ \E\left[ |\{ n \ge 0:~ Z_n(Kt^{-1}) =K \}|\right] }{\E\left[ |\{ n \ge 0:~ Z_n(K) =K \}|\right] }.  
\end{eqnarray*}
Since this is true for every $E \subset B_{Ks}$ it follows that
$$ \frac{d\nu_{Kt^{-1}}}{d\nu_K}(\xi) = \frac{ \E\left[ |\{ n \ge 0:~ Z_n(Kt^{-1}) =K \}|\right] }{\E\left[ |\{ n \ge 0:~ Z_n(K) =K \}|\right] }.$$
Because we have assumed the random walk on $K\backslash G$ is transient, the expected values appearing in the formulae above are finite. This shows that, $\frac{d\nu_{Kt^{-1}}}{d\nu_K}(\xi) =\frac{ \nu_{Kt^{-1}}(E) }{ \nu_K(E)}$ for every measurable $E \subset B_{Ks}$ with positive measure. This proves the lemma in the case $Kt^{-1} \ne Ks$.

Suppose now that $Kt^{-1}=Ks$. Then any path in $\Schreier(K \backslash G)$ from $K$ which projects into $B_{Ks}$ must pass through $Ks$. So for any Borel subset $E \subset B_{Ks}$, 
\begin{eqnarray*}
\nu_{K}(E) &=& \Prob(\zeta(K) \in E)\\
 & = &\sum_{n=0}^\infty \Prob\left(Z_n(K) = Ks \right)\cdot \Prob\left(\zeta(Ks) \in E,~Z_r(Ks) \ne Ks ~\forall r>0\right)\\
 &=& \E\left[ |\{ n \ge 0:~ Z_n(K) =Ks \}|\right] \cdot \Prob \left(\zeta(Ks) \in E,~ Z_r(Ks) \ne Ks ~\forall r>0\right)\\
  &=& \E\left[ |\{ n \ge 0:~ Z_n(K) =Ks \}|\right] \frac{\Prob \left(\zeta(Ks) \in E \right) }{\sum_{n=0}^\infty \Prob(Z_n(Ks)=Ks)}\\
  &=& \nu_{Ks}(E) \frac{ \E\left[ |\{ n \ge 0:~ Z_n(K) =Ks \}|\right] }{\E\left[ |\{ n \ge 0:~ Z_n(Ks) =Ks \}|\right] }.  
\end{eqnarray*}
Since this is true for every $E \subset B_{Ks}$ it follows that
$$  \frac{d\nu_{Kt^{-1}}}{d\nu_K}(\xi) =\frac{d\nu_{Ks}}{d\nu_K}(\xi) =  \frac{ \E\left[ |\{ n \ge 0:~ Z_n(Ks) =Ks \}|\right] }{\E\left[ |\{ n \ge 0:~ Z_n(K) =Ks \}|\right] }.$$
In particular, $\frac{d\nu_{Kt^{-1}}}{d\nu_K}(\xi) =\frac{ \nu_{Kt^{-1}}(E) }{ \nu_K(E)}$ for every Borel $E \subset B_{Ks}$ with positive measure. This proves the lemma in the case $Kt^{-1} = Ks$.
\end{proof}

 
\begin{lem}\label{lem:rn3}
For any $s \in \{a,b,a^{-1},b^{-1}\}$ and $K \in \Tree_G$, 
$$1/4 \le\frac{d\nu_{Ks}}{d\nu_K}  \le 4$$
almost everywhere.
\end{lem}
\begin{proof}
For $n\ge 0$ and $g\in G$, define $Z_n(Kg)$ and $\zeta(Kg)$ as in the proof of the previous lemma. For any Borel set $E \subset B_K$,
\begin{eqnarray*}
\nu_{Ks}(E) = \Prob(\zeta(Ks) \in E) \ge \Prob\left(Z_1(Ks) = K\right)\cdot \Prob( \zeta(K) \in E) \ge  \nu_K(E) / 4.
\end{eqnarray*}
The other inequality is similar.
\end{proof}

\begin{proof}[Proof of Theorem \ref{thm:continuous}]
For $K \in \Sub_G$, let $S(K) = \{Ka, Ka^{-1}, Kb, Kb^{-1}\}$. Because it might occur that $Ka=Kb$, for example, it is possible that $|S(K)|<4$. By Lemma \ref{lem:rn2}, for any $\lambda \in \cM(\Tree_G)$,
\begin{eqnarray*}\label{eqn:thmc}
h_\mu(\lambda) &=& -\iint \log \frac{d\nu_\lambda \circ g}{d\nu_\lambda}(\xi)~d\nu_\lambda(\xi)~d\mu(g)\\
&=& -\iint \sum_{Ks \in S(K)}  \nu_K(B_{Ks}) \log \frac{\nu_{Kg^{-1}}(B_{Ks})}{\nu_K(B_{Ks})}~d\lambda(K)d\mu(g)\\
&=& -\int \sum_{t\in G} \sum_{Ks\in S(K)} \mu(t) \nu_K(B_{Ks}) \log \frac{\nu_{Kt^{-1}}(B_{Ks})}{\nu_K(B_{Ks})}~d\lambda(K).
\end{eqnarray*}
By the previous lemma there is a constant $C>0$ so that 
$$\left| \nu_K(B_{Ks}) \log \frac{\nu_{Kt^{-1}}(B_{Ks})}{\nu_K(B_{Ks})} \right| \le C$$
for all $K \in \Tree_G$ and $s,t \in \{e,a,b,a^{-1},b^{-1}\}$. For $x,y \in [0,1]$ let
\begin{displaymath}
F(x,y) := \left\{  \begin{array}{cc}
-C &  - x\log \frac{y}{x} \le -C\\
 - x\log \frac{y}{x} &  -C<- x\log \frac{y}{x} < C \\
C &  - x\log \frac{y}{x} \ge C
\end{array}\right.
\end{displaymath}
Also for $t \in \{a,b,a^{-1},b^{-1},e\}$, $K\in \Sub_G$ and $Ks \in S(K)$, let
\begin{eqnarray*}
\rho(K,Ks,t)&:=&\mu(t)F( \nu_{K}(B_{Ks}), \nu_{Kt^{-1}}(B_{Ks}) ).
\end{eqnarray*}
So the previous equation implies
 \begin{eqnarray}\label{eqn:thmc1}
h_\mu(\lambda)=  \int \sum_{t\in G}\sum_{Ks\in S(K)}   \rho(K,Ks,t)~d\lambda(K).
\end{eqnarray}

Define $Z_n(Kg)$ and $\zeta(Kg)$ as in Lemma \ref{lem:rn2}. For $n,\epsilon \ge 0$, $t \in \{a,b,a^{-1},b^{-1},e\}$, $K\in \Sub_G$ and $Ks \in S(K)$, let 
\begin{eqnarray*}
\rho_n(K,Ks,t)&:=& \mu(t) F\left(\Pr\left(Z_n(K) \in \shadow(Ks)\right), \Pr\left(Z_n(Kt^{-1}) \in \shadow(Ks)\right)\right).
\end{eqnarray*}

Note that $\rho_n(K,Ks,t)$ varies continuously with $K \in \Tree_G$ (for fixed $s,t,n$). Since we are using the weak* topology on $\cN \subset \cM(\Sub_G)$, it suffices to show that
$$\lim_{n\to\infty} \sup_{\lambda \in \cN} \left| h_\mu(B(\Sub_G), \nu_\lambda) -   \int  \sum_{t\in G} \sum_{Ks\in S(K)} \rho_n(K,Ks,t)~d\lambda(K)\right| = 0.$$
Let $S=\{a,b,a^{-1},b^{-1}\}$. By (\ref{eqn:thmc1}) it suffices to prove
$$\lim_{n\to\infty} \sup_{\lambda \in \cN}  \sum_{t\in S} \sum_{s\in S}  \int \left|\rho(K,Ks,t) -  \rho_n(K,Ks,t)\right|~d\lambda(K) = 0.$$
Let
$$X_{n,s,t}(\epsilon) := \{K \in \Sub_G:~\left|\rho(K,Ks,t) -  \rho_n(K,Ks,t)\right| < \epsilon\}.$$
Because $\left|\rho(K,Ks,t) -  \rho_n(K,Ks,t)\right|$ is bounded by $2C$, it suffices to show that
$$ \lim_{n\to\infty} \sup_{\lambda \in \cN} \lambda(X_{n,s,t}(\epsilon) ) = 1\quad \forall \epsilon>0, s,t\in S.$$
For $\delta>0$, let
$$Y_{n,s,t}(\delta) := \left\{K \in \Sub_G:~\left| \nu_{Kt^{-1}}(B_{Ks})  - \Pr\left(Z_n(Kt^{-1}) \in \shadow(Ks)\right)\right| < \delta \right\}.$$
Because $F$ is uniformly continuous on $[0,1] \times [0,1]$, it suffices to prove that
\begin{eqnarray}\label{eqn:thmc2}
 \lim_{n\to\infty} \sup_{\lambda \in \cN} \lambda(Y_{n,s,t}(\delta) ) = 1 \quad \forall \delta>0, s,t \in S.
\end{eqnarray}

Because $\cN$ has controlled return-time probabilities, for any $\delta>0$ there exists an $N=N(\delta)$ such that $n \ge N$ implies
$$\lambda \Big( \{K \in \Sub_G:~R_{\ge n}(\mu,K) \ge \delta\} \Big) < \delta~\forall \lambda \in \cN.$$
Equivalently,
$$\lambda \Big( \left\{K \in \Sub_G:~\Pr\left(Z_n(K) \ne K,~ \forall n \ge N\right)  \ge 1-\delta \right\} \Big) \ge 1- \delta~\forall \lambda \in \cN.$$

Because the support of $\mu$ is contained in $\{ a,b,a^{-1},b^{-1}\}$ if $Z_n(K) \ne K$ for any $n \ge N$ and $Z_n(K) \in \shadow(Ks)$ (for some $s \in \{a,b,a^{-1},b^{-1}\}$) then $\zeta(K) \in B_{Ks}$. Therefore, the equation above implies that for any $s\in \{a,b,a^{-1},b^{-1}\}$ and any $n\ge N$,
$$\lambda \Big( \left\{K \in \Sub_G:~ \left| \Pr(Z_n(K) \in \shadow(Ks)) - \Pr(\zeta(K) \in B_{Ks}) \right| \le \delta\right\} \Big) \ge 1- \delta,~\forall \lambda \in \cN.$$
Since $\Pr(\zeta(K)\in B_{Ks}) = \nu_K(B_{Ks})$, this equation is equivalent to
$$\lambda \Big( \left\{K \in \Sub_G:~ \left| \Pr\left(Z_n(K) \in \shadow(Ks)\right) -\nu_K(B_{Ks}) \right| \le \delta\right\} \Big) \ge 1- \delta,~\forall \lambda \in \cN.$$
This implies equation (\ref{eqn:thmc2})  for $t=e$. The other cases are similar.

\end{proof}


\subsection{A covering space construction}\label{sec:cover}

For $K \in \Tree_G$, let $X_K$ be the 2-complex whose 1-skeleton is the right-Schreier coset graph of $K\backslash G$ and whose 2-cells are all possible 1-gons and 2-gons. More precisely, for every loop in the Schreier coset graph, there is a 2-cell whose boundary is that loop and if $e_1,e_2$ are two edges with the same endpoints, then there is a 2-cell with boundary $e_1 \cup e_2$. Because $K \backslash G$ is tree-like, $X_K$ is simply-connected.


If $c$ is a 2-cell of $X_K$ and $g \in G$, then we let $gc$ be the corresponding 2-cell of $X_{gKg^{-1}}$. For example, if $c$ is bounds a loop based at the vertex $Kh \in X^{(0)}_K$ then $gc$ bounds a loop based at the vertex $gKh \in X^{(0)}_{gKg^{-1}}$. If $c$ bounds a pair of edges $(Kh, Khs), (Kh, Kht)$ (for some $t,s \in S:=\{a,b,a^{-1},b^{-1}\}$) with the same endpoints, then $gc$ bounds the pair of edges $(gKh, gKhs), (gKh, gKht)$. 

Let $\overline{\Tree_G}$ be the set of all pairs $(K,\omega)$ where $K \in \Tree_G$ and $\omega \subset X_K^{(2)}$ is a collection of $2$-cells of $X_K$. $G$ acts on this space by $g(K,\omega) = (gKg^{-1}, g\omega)$ where $g\omega =\{gc:~c \in \omega\}$.


There is a natural topology on $\overline{\Tree_G}$ which works by identifying $(K,\omega) \in \overline{\Tree_G}$ with the (labeled) complex $X_K \setminus \omega$. To explain, let $B_n(X_K)$ denote the subcomplex of $X_K$ consisting of all cells $c$ such that every vertex $v$ incident to $c$ has distance at most $n$ from $K$ with respect to the path metric on the 1-skeleton of $X_K$. For each integer $n\ge 1$ and $(K,\omega) \in \overline{\Tree_G}$, let $\textrm{Nbhd}_n(K,\omega)$ be the set of all $(K',\omega') \in \overline{\Tree_G}$ such that there is a cell-complex isomorphism $\phi:B_n(X_K) \to B_n(X_{K'})$ which preserves edge-labels and directions on the edges and also maps $B_n(X_K) \cap \omega$ bijectively onto $B_n(X_{K'}) \cap \omega'$. We obtain a topology on $\overline{\Tree_G}$ by declaring that each $\textrm{Nbhd}_n(K,\omega)$ is clopen. 

This topology makes $\overline{\Tree_G}$ a compact metrizable space. Indeed, for any fixed $n$ there are only finitely many subsets of the form $\textrm{Nbhd}_n(K,\omega)$ (and these are pairwise disjoint). Thus if $\{(K_i,\omega_i)\}_{i=1}^\infty$ is a sequence in $\overline{\Tree_G}$, then after passing to a subsequence we may assume the existence of $(K,\omega) \in \overline{\Tree_G}$ such that any $n$, for all sufficiently large $i$, $(K_i,\omega_i) \in \textrm{Nbhd}_n(K,\omega)$. It follows that $\lim_{i\to\infty} (K_i,\omega_i) = (K,\omega)$ and thus  $\overline{\Tree_G}$ is compact. We can also define a metric on $\overline{\Tree_G}$ by declaring the distance between any  $(K_1,\omega_1), (K_2,\omega_2) \in \overline{\Tree_G}$ to be $\frac{1}{n+1}$ where $n$ is the largest natural number with $(K_i,\omega_i) \in \textrm{Nbhd}_n(K_{3-i},\omega_{3-i})$ for $i=1,2$.
 
Let $(K,\omega) \in \overline{\Tree_G}$. We write $X_K\setminus \omega$ to mean the subcomplex of $X_K$ which contains the 1-skeleton of $X_K$ and every 2-cell other than those in $\omega$. Let $\pi_{K,\omega}:U_{K,\omega} \to X_K \setminus \omega$ denote the universal cover. There is a natural transitive right-action of $G$ on $U_{K,\omega}$ obtained by realizing the 1-skeleton of $U_{K,\omega}$ as the Schreier coset graph of a subgroup of $G$. To be precise, if $v$ is a vertex of $U_{K,\omega}$ and $s \in \{a,b,a^{-1},b^{-1}\}$ then $vs$ is the vertex such that $(v,vs)$ is an edge of $U_{K,\omega}$ and $\pi_{K,\omega}(v,vs)$ is labeled $s$. 

Choose a vertex $u_{K,\omega} \in U^{(0)}_{K,\omega}$ such that $\pi_{K,\omega}(u_{K,\omega})=K$ and let $S_{K,\omega}$ be the stabilizer $S_{K,\omega}:=\{g \in G:~ u_{K,\omega} g = u_{K,\omega}\}$. Because $S_{K,\omega}\backslash G$ is naturally identified with the 1-skeleton of $U_{K,\omega}$, it follows that $S_{K,\omega} \backslash G$ is tree-like. Also, observe that $S_{K,\omega}$ does not depend on the choice of $u_{K,\omega}$. Indeed, it is the subgroup of $K$ generated by all elements of the form 
\begin{enumerate}
\item $gsg^{-1}$ for every 2-cell {\bf not} in $\omega$ which bounds a loop based at $Kg$ labeled $s \in \{a,b,a^{-1},b^{-1}\}$;
\item $g s_1 s_2^{-1} g^{-1}$ for every 2-cell {\bf not} in $\omega$ which bounds a bigon whose edges are labeled $s_1,s_2 \in \{a,b,a^{-1},b^{-1}\}$ and are directed from $Kg$ to $Kgs_1=Kgs_2$.
\end{enumerate}
The fundamental group of $X_K \setminus \omega$ is $K/S_{K,\omega}$. In particular, $S_{K,\omega}$ is normal in $K$. 

The map from $\Psi:\overline{\Tree_G} \to \Tree_G$ defined by $\Psi(K,\omega)= S_{K,\omega}$ is $G$-equivariant. Therefore, if $\widetilde{\eta}$ is a $G$-invariant ergodic probability measure on $\overline{\Tree_G}$, then, $\Psi_*\widetilde{\eta}$ is $G$-invariant and ergodic.  Also $\Psi$ is continuous, so $\Psi_*: \cM(\overline{\Tree_G}) \to \cM(\Tree_G)$ is continuous in the weak* topology where $\cM(\overline{\Tree_G})$ denotes the space of $G$-invariant Borel probability measures on $\overline{\Tree_G}$.

\begin{lem}\label{lem:R-monotone}
For any subgroup $K<G$, let $\overline{R}(K)$ be the expected number of returns of the random walk on $K\backslash G$ to $K$. That is,
$$\overline{R}(K) = \int \#\{ n :~ Kg_n =K\}~d\P(g_0,g_1,\ldots).$$
For $\lambda \in \cM(\Sub_G)$, let $\overline{R}(\lambda) = \int \overline{R}(K)~d\lambda(K)$. If $\widetilde{\lambda}$ is a $G$-invariant measure on $\overline{\Tree_G}$ which projects to $\lambda \in \cM(\Sub_G)$, then  $\overline{R}(\lambda) \ge \overline{R}(\Psi_* \widetilde{\lambda}).$
\end{lem}

\begin{proof}
Because $\pi_{K,\omega}:U_{K,\omega} \to X_K\setminus \omega$ is a covering map; the expected number of times a random walk started at a vertex $v$ in $U_{K,\omega}$ returns to $v$ is bounded by the expected number of times the projected random walk returns to $\pi_{K,\omega}(v)$. So $\overline{R}(S_{K,\omega}) \le \overline{R}(K)$ and 
\begin{eqnarray*}
\overline{R}(\Psi_* \widetilde{\lambda}) = \int \overline{R}(S_{K,\omega}) ~d\tilde{\lambda}(K,\omega) \le \int \overline{R}(K)~d\lambda = \overline{R}(\lambda).
\end{eqnarray*}
\end{proof}

\begin{cor}\label{cor:cont}
For $\eta \in \cM(\Tree_G)$, let $\cM_\eta(\overline{\Tree_G})  $ be the space of all $G$-invariant Borel probability measures on $\overline{\Tree_G}$ which project to $\eta$. This is a compact convex space under the weak* topology. Moreover, if $\overline{R}(\eta)<\infty$ then the map which sends $\lambda \in \overline{\Tree_G}$ to $h_\mu(\Psi_*\lambda)$ is continuous on $\cM_\eta(\overline{\Tree_G})  $. 
\end{cor}

\begin{proof}
By the previous lemma, $\infty>\overline{R}(\eta)\ge \overline{R}(\Psi_*\lambda)$ for all $\lambda \in \cM_\eta(\overline{\Tree_G})  $. Therefore, $\cM_\eta(\overline{\Tree_G})  $ has controlled return-time probabilities. So the corollary follows from Theorem \ref{thm:continuous}.
\end{proof}


In order to prove Theorem \ref{thm:main}, it now suffices to show there exists a sequence $\{\eta_n\}_{n=1}^\infty \subset \cM(\Tree_G)$ and for every $n$, a continuous 1-parameter family $\{\eta_{n,p} :~ 0\le p \le 1\} \subset \cM_{\eta_n}(\overline{\Tree_G}) $ of ergodic measures such that $\overline{R}(\eta_n)<\infty$ for all $n$, $\lim_{n \to \infty} h_\mu(\eta_n) = 0$, $\Psi_*(\eta_{n,0})=\eta_n$ and $\Psi_*(\eta_{n,1})$ is the Dirac measure on the trivial subgroup. This is accomplished in the next section.

\subsection{Paths of IRS's}

For each integer $n\ge 1$, we define a subgroup $K_n < G$ as follows (see figure \ref{fig:3} for an example). $K_n$ is generated by all elements of the form $g h g^{-1}$ where $g \in \langle a^n, b^n\rangle$ and either $h = a^k b^r a^{-k}$ for some $1 \le |k| \le n-1$ and $r \in \Z$ or $h=b^k a^r b^{-k}$ for some $1 \le |k| \le n-1$ and $r \in \Z$.

\begin{figure}[htb]
\begin{center}
\ \psfig{file=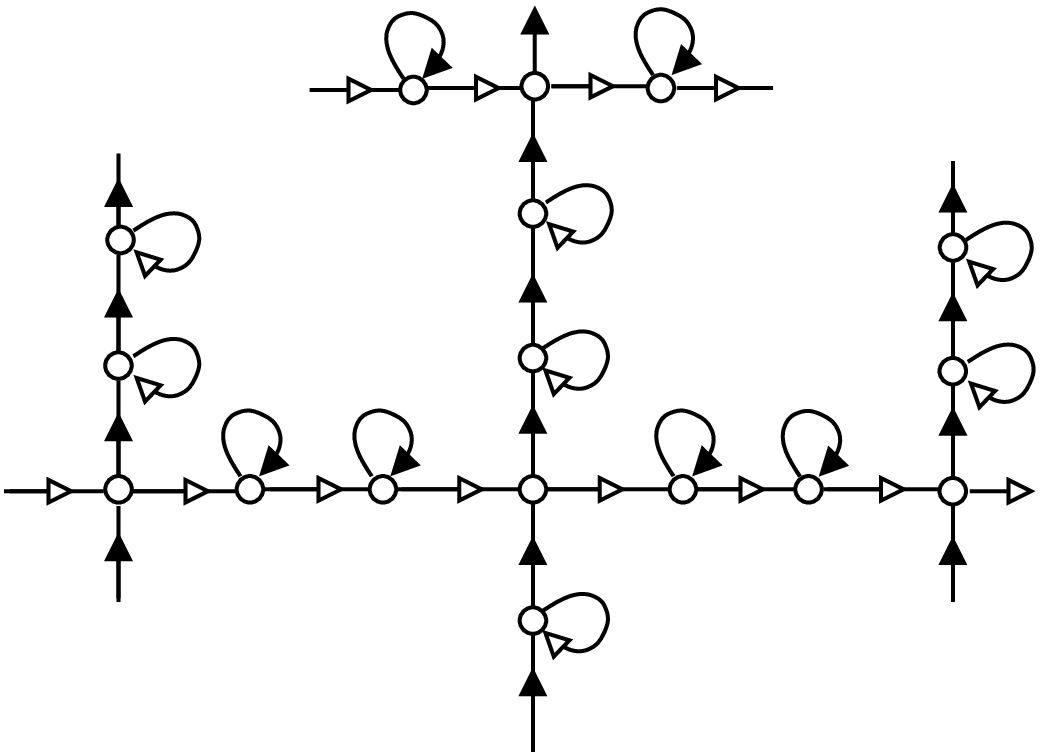,height=2.5in,width=4in}
\caption{Part of the Schreier coset graph of $K_3\backslash G$. The black arrows represent $a$ and the white arrows represent $b$. }
\label{fig:3}
\end{center}
\end{figure}

Note that there are only a finite number of $G$-conjugates of $K_n$. Indeed, 
$$\cC_n:=\{a^{-i} K_n a^{i}, b^{-i}K_n b^{i}:~ 0\le i \le n-1\}$$
 is a complete set of conjugates. To see this, it suffices to show that for every group $J\in \cC_n$ and $s \in \{a,b,a^{-1},b^{-1}\}$, $sJs^{-1}\in \cC_n$.  For example, note that for any $1\le i \le n-1$, $a^ib^{-1} a^{-i}, a^{i}ba^{-i} \in K_n$. Therefore, 
$$ba^{-i} K_n a^i b^{-1} = ba^{-i} (a^ib^{-1} a^{-i}) K_n (a^{i}ba^{-i}) a^i b^{-1} = a^{-i} K_n a^{i} \in \cC_n.$$
The other cases are similar.
Let $\eta_n \in \cM(G)$ be the measure uniformly distributed on $\cC_n$, the set of $K_n$ conjugates.  From figure \ref{fig:3}, it is apparent that $K_n \backslash G$ is tree-like. So $\eta_n \in \cM(\Tree_G)$.



\begin{cor}\label{cor:eta-n}
The map $\lambda \in \cM_{\eta_n}(\overline{\Tree_G})  \mapsto h_\mu(\Psi_*\lambda )$ is continuous on $\cM_{\eta_n}(\overline{\Tree_G})  $.
\end{cor}
\begin{proof}
It is not difficult to see that $\overline{R}(\eta_n)$ is finite. For example, this can be seen from the obvious quasi-isometry between $K_n\backslash G$ and the free group $G$ or from the classification of recurrent quasi-transitive graphs (as explained in \cite{Wo00} for example). The present corollary now follows from Corollary \ref{cor:cont}.
\end{proof}

\begin{lem}\label{lem:limits}
Let $\eta_{n,1} \in \cM_\eta(\overline{\Tree_G})  $ be uniformly distributed on the $G$-orbit of $(K_n, X^{(2)}_{K_n})$. Then $\Psi_*\eta_{n,1}$ is the trivial subgroup, so $h_\mu(\Psi_*\eta_{n,1}) = h_{max}(\mu)$.  Also $\lim_{n\to \infty} h_\mu(\Psi_*\eta_n) = 0$.
\end{lem}
\begin{proof}
The first claim is obvious. Note $\eta_n$ converges in the weak* topology to $\kappa=(1/2)\delta_A + (1/2)\delta_B$ (as $n\to\infty$) where $A$ is the smallest normal subgroup of $G$ containing $\{a^n:~n\in \Z\}$ and $B$ is the smallest normal subgroup of $G$ containing $\{b^n:~n\in \Z\}$. Both $G/A$ and $G/B$ are isomorphic to the group of integers. Because the random walk on $\Z$ has zero entropy, $h_\mu( B(\Sub_G), \nu_{\kappa}) =0$. By Theorem \ref{thm:hard}, $\lambda \mapsto h_\mu(B(\Sub_G), \nu_\lambda)$ is an infimum of continuous functions and is therefore, upper semi-continuous. Therefore, 
$$\limsup_{n\to\infty} h_\mu( \Psi_*\eta_{n,0})  \le h_\mu( B(\Sub_G), \nu_{\kappa}) =0.$$
\end{proof}

Let  $\eta_{n,0}$ be the measure uniformly distributed on the $G$-orbit of $(K_n, \emptyset)$. Trivially, $\Psi_*(\eta_{n,0}) =\eta_n$. Because of the Lemma above and Corollary \ref{cor:eta-n} to prove Theorem \ref{thm:main} it now suffices to show that for every $n\ge 1$ there exists a continuous path of ergodic measures in $\cM_\eta(\overline{\Tree_G})  $ from $\eta_{n,1}$ to $\eta_{n,0}$. We give two different proofs of this fact. The first is constructive. The second proof (in the next section) shows that in fact the entire space $\cM^e_{\eta_n}(\overline{\Tree_G}) $ of ergodic measures in $\cM_{\eta_n}(\overline{\Tree_G}) $ is pathwise connected.


\begin{proof}[Proof of Theorem \ref{thm:main}]
Let $0\le p \le 1$ and $n\ge 1$ be an integer. Let $K'$ be a uniformly random conjugate of $K_n$. Let $\omega$ be the random subset of $X_{K'}^{(2)}$ satisfying
\begin{itemize}
\item for every disjoint pair of finite sets $Y,Z \subset X_K^{(2)}$, the probability that $Y \subset \omega$ and $Z \cap \omega = \emptyset$ is $p^{|Y|}(1-p)^{|Z|}$.
\end{itemize}
Let $\eta_{n,p}$ be the law of $(K',\omega)$. It is a $G$-invariant ergodic probability measure on $\overline{\Tree_G}$. Also, $p \mapsto \eta_{n,p}$ is continuous. So Corollary \ref{cor:eta-n} implies $p \mapsto h_\mu( \Psi_*\eta_{n,p})$ is continuous. By Lemma \ref{lem:limits}, for every $t$ with $ h_\mu(\eta_n) \le t \le h_{max}(\mu)$, there is a $p \in [0,1]$ such that $h_\mu(\Psi_*\eta_{n,p})=t$. Because $\lim_{n\to \infty} h_\mu(\eta_n) = 0$, this implies the theorem.
\end{proof}

It may interest the reader to know that the paths $p \mapsto h_\mu(\Psi_*\eta_{n,p})$ are monotone increasing. This follows from the next lemma and corollary.

\begin{lem}
Let $\rho$ be a Borel probability measure on $\{ (K_1,K_2) \in \Sub_G\times \Sub_G:~ K_1 \le K_2\}$. Suppose $\rho$ is invariant under the diagonal action of $G$ by conjugation. For $i=1,2$, let $\rho_i$ be the projection of $\rho$ onto the $i$-th coordinate. Then $h_\mu(B(\Sub_G), \nu_{\rho_1}) \ge h_\mu(B(\Sub_G), \nu_{\rho_2}) $.
\end{lem}
\begin{proof}
Observe that if $K_1 \le K_2$ then $H(\mu^n_{K_1}) \ge H(\mu^n_{K_2})$ since the projection map $K_1 \backslash G \to K_2 \backslash G_2$ maps $\mu^n_{K_1}$ onto $\mu^n_{K_2}$. By Theorem \ref{thm:hard},
\begin{eqnarray*}
 h_\mu(B(\Sub_G), \nu_{\rho_1}) &=& \lim_{n\to\infty} \frac{1}{n} \int H(\mu^n_K)~d\rho_1(K) = \lim_{n\to\infty} \frac{1}{n} \int H(\mu^n_{K_1})~d\rho(K_1,K_2) \\
 &\ge& \lim_{n\to\infty} \frac{1}{n} \int H(\mu^n_{K_2})~d\rho(K_1,K_2) = \lim_{n\to\infty} \frac{1}{n} \int H(\mu^n_K)~d\rho_2(K) \\
  &=&  h_\mu(B(\Sub_G), \nu_{\rho_2}).
   \end{eqnarray*}
   \end{proof}

\begin{cor}
The paths $p \mapsto h_\mu(\Psi_*\eta_{n,p})$ are monotone increasing. 
\end{cor}
\begin{proof}
Let $K \in \Tree_G$ be random with law $\eta_n$. Let Leb denote Lebesgue measure on $[0,1]$ and let $x:X_K^{(2)} \to [0,1]$ be random with law $\textrm{Leb}^{X_K^{(2)}}$. In other words, for each cell $c \in X_K^{(2)}$, $x(c)$ has law Leb and the variables $\{x(c):~c \in X_K^{(2)}\}$ are independent.

 Fix $p,q$ with $0\le p \le q \le 1$. Let $\omega_p=x^{-1}( [0,p] )$ and $\omega_q = x^{-1}( [0,q] )$. Let $\rho$ be the law of the pair $(S_{K,\omega_q}, S_{K,\omega_p})$ (where $S_{K,\omega}$ is defined in \S \ref{sec:cover}). Clearly the projection of $\rho$ onto its first factor is $\Psi_*\eta_{n,q}$ and the projection onto its second factor is $\Psi_*\eta_{n,p}$. Because $\omega_p \subset \omega_q$, it follows that $S_{K,\omega_q} < S_{K,\omega_p}$. So the previous lemma implies $h_\mu(\Psi_*\eta_{n,q}) \ge h_\mu(\Psi_*\eta_{n,p}) $ as required.
 \end{proof}
 
\subsection{Entropies of boundaries of quotients by almost normal subgroups}


In this subsection, we prove Theorem \ref{thm:dense}. We need a few lemmas first.


\begin{lem}\label{lem:equi}
Let $\Omega_n$ be the set of all subsets of $\{a^1,\ldots, a^{n-1}, b^1,\ldots, b^{n-1}\}$ and $\Gamma_n=\langle a^n,b^n\rangle<G$. Let $\Gamma_n$ act on $\Omega_n^{\Gamma_n}$ by $(gx)(f)=x(g^{-1}f)$. Then there is a $\Gamma_n$-equivariant homeomorphism from $\Omega_n^{\Gamma_n}$ to the set $\{ (K_n,\omega)\in \Tree_G:~ \omega \subset X^{(2)}_{K_n}\}$. 
\end{lem}

\begin{proof}
First observe that for every 2-cell $\omega$ in $X_{K_n}^{(2)}$ there is a unique $g \in \Gamma_n$ and $i$ with $1\le i \le n-1$ such that $\omega$ is bounded by either the loop with vertex $K_n ga^i$ or the loop with vertex $K_ngb^i$.

Define $\Phi:\Omega_n^{\Gamma_n} \to \{ (K_n,\omega)\in \Tree_G:~ \omega \subset X^{(2)}_{K_n}\}$ by: $\Phi(x) = (K_n,\omega)$ where $\omega$ consists of all 2-cells whose boundary consists of a loop based at $K_nga^i$ or $K_ngb^j$ for $g \in \Gamma_n$, $1\le i,j \le n-1$ and $a_i,b_j \in x(g)$. It is routine to check that $\Phi$ is the required $\Gamma_n$-equivariant homeomorphism.
\end{proof}

\begin{lem}\label{lem:EMD}
Let $\cM_{\Gamma_n}(\Omega_n^{\Gamma_n})$ be the space of all $\Gamma_n$-invariant Borel probability measures on $\Omega_n^{\Gamma_n}$. Let $\cM^p_{\Gamma_n}(\Omega_n^{\Gamma_n})$ be the set of all those measures $\mu \in \cM_{\Gamma_n}(\Omega_n^{\Gamma_n})$ that are ergodic and have finite support. Then $\cM^p_{\Gamma_n}(\Omega_n^{\Gamma_n})$ is dense in 
$\cM_{\Gamma_n}(\Omega_n^{\Gamma_n})$.
\end{lem}

\begin{proof}
This is an immediate consequence of \cite[Theorem 1]{Ke12}, which states that free groups have property EMD (i.e., their profinite completions weakly contain all measure-preserving actions). The fact that property EMD implies the property of this lemma (that $\cM^p_{\Gamma_n}(\Omega_n^{\Gamma_n})$ is dense in 
$\cM_{\Gamma_n}(\Omega_n^{\Gamma_n})$) is an easy exercise contained in \cite[Proposition 3.5]{TD12}.  It had earlier been proven by the author in \cite[Theorem 3.4]{Bo03} that subset of measures $\mu$ that have finite support is dense in $\cM_{\Gamma_n}(\Omega_n^{\Gamma_n})$.
\end{proof}

\begin{lem}\label{lem:simplex}
Let $\cM^p_{\eta_n}(\overline{\Tree_G})$ be the set of all $G$-invariant Borel probability measures on $\overline{\Tree_G}$ which project to $\eta_n$, are ergodic and have finite support. Then $\cM^p_{\eta_n}(\overline{\Tree_G})$  is dense in $\cM_{\eta_n}(\overline{\Tree_G})$ 
\end{lem}

\begin{proof}
By Lemmas \ref{lem:equi} and \ref{lem:EMD}, the set $\cM^p_{\Gamma_n}(K_n)$ of ergodic $\Gamma_n$-invariant Borel probability measures on $\{ (K_n,\omega)\in \Tree_G:~ \omega \subset X^{(2)}_{K_n}\}$ which have finite support is dense in the space $\cM_{\Gamma_n}(K_n)$ of all $\Gamma_n$-invariant Borel probability measures on $\{ (K_n,\omega)\in \Tree_G:~ \omega \subset X^{(2)}_{K_n}\}$. The map $\Lambda: \cM_{\Gamma_n}(K_n) \to \cM_{\eta_n}(\overline{\Tree_G})$ defined by
$$\Lambda(\lambda) = \frac{1}{2n-1}\left(\lambda + \sum_{i=1}^{n-1} a^i_*\lambda + b^i_*\lambda\right)$$
is an affine isomorphism which maps $\Gamma_n$-ergodic measures to $G$-ergodic measures and measures with finite support to measures with finite support. It therefore maps $\cM^p_{\Gamma_n}(K_n)$ to $\cM^p_{\eta_n}(\overline{\Tree_G})$, proving that the latter is dense in $\cM_{\eta_n}(\overline{\Tree_G})$.
\end{proof}

As promised we can now prove that $\cM^e_{\eta_n}(\overline{\Tree_G})$ is pathwise connected. For this, recall that a convex closed metrizable subset $\cK$ of a locally convex linear
space is a simplex if each point in $\cK$ is the barycenter of a unique probability
measure supported on the subset $\partial_e \cK$ of extreme points of  $\cK$. In this case, $\cK$ is called a {\em Poulsen simplex} if $\partial_e \cK$ is dense in $\cK$. It is known from \cite{LOS78} that there is a unique Poulsen simplex up to affine isomorphism. Moreover, its set of extreme points is homeomorphic to $l^2$. The previous lemma implies:


\begin{cor}
For each $n \ge 1$, $\cM_{\eta_n}(\overline{\Tree_G})$ is a Poulsen simplex. Therefore, the subspace of ergodic measures $\cM^e_{\eta_n}(\overline{\Tree_G}) \subset \cM_{\eta_n}(\overline{\Tree_G})$ is homeomorphic to the Hilbert space $l^2$. In particular, $\cM^e_{\eta_n}(\overline{\Tree_G})$ is pathwise connected.
\end{cor}

\begin{proof}[Proof of Theorem \ref{thm:dense}]
By Corollary \ref{cor:eta-n} and Lemma \ref{lem:simplex}, for every $n>0$, the set of all numbers $t$ such that $t = h_\mu( \nu_{\Psi_*\lambda})$ for some ergodic $\lambda \in \cM_{\eta_n}(\overline{\Tree_G})$ with finite support is dense in $[h_\mu(\eta_n), h_{max}(\mu)]$. If $\lambda \in \cM_{\eta_n}(\overline{\Tree_G})$ is ergodic with finite support then $\Psi_*\lambda$ is ergodic with finite support. Ergodicity implies $\Psi_*\lambda$ is supported on a single finite conjugacy class which implies, by invariance, that it is the uniform probability measure on a single conjugacy class. In other words, there is a conjugacy class $\{L_1,\ldots, L_n\} \subset \Sub_G$ such that $\Psi_*\lambda = \frac{1}{n} \sum_{i=1}^n \delta_{L_i}$. By Lemma \ref{lem:limits}, $\lim_{n\to\infty} h_\mu(\eta_n) = 0$. This implies the theorem.
\end{proof}

\end{document}